\documentclass{amsart}

\usepackage{amsmath, amsthm,amssymb, amsfonts}
\usepackage{rotating}
\usepackage{url}
\usepackage{microtype}
\usepackage{cite}

\def\E{{\mathbb E}}

\def\ER{{Erd\H{o}s--R\'{e}nyi }}

\renewcommand{\phi}{\varphi}

\newcommand{\comment}[1]{}

\def\Z{{\mathbb Z}}

\def\P{{\mathbb P}}

\def\height{{\operatorname{height}}}
\def\width{{\operatorname{width}}}
\def\corank{{\operatorname{corank}}}

\def\rank{{\operatorname{rank}}}

\def\deg{{\operatorname{deg}}}

\newtheorem*{Berry-Esseen}{Berry-Esseen Theorem}

\newtheorem{thm}{Theorem}

\newtheorem{prop}[thm]{Proposition}
\newtheorem{cor}[thm]{Corollary}
\newtheorem{lemma}[thm]{Lemma}
\newtheorem*{defn}{Definition}

\newtheorem{claim}{Claim}

\theoremstyle{remark}

\title{Sandpile Groups of Random Bipartite Graphs}

\author{Shaked Koplewitz}
\address{Mathematics Department, Yale University, New Haven, CT 06511}
\email{shaked.koplewitz@gmail.com}
\begin{document}

\begin{abstract}

We determine the asymptotic distribution of the $p$-rank of the sandpile groups of random bipartite graphs. We see that this depends on the ratio between the number of vertices on each side, with a threshold when the ratio between the sides is equal to $\frac{1}{p}$.
We follow the approach of Wood in \cite{mw} and consider random graphs as a special case of random matrices, and rely on a variant the definition of min-entropy given by Maples in \cite{km} in order to obtain useful results about these random matrices. Our results show that unlike the sandpile groups of \ER random graphs, the distribution of the sandpile groups of random bipartite graphs depends on the properties of the graph, rather than coming from some more general random group model.

\end{abstract}

\maketitle

\section{Introduction}
\label{sec:intro}

\subsection{The Main Theorem}

In this paper, we study the sandpile group of a random bipartite graph. Recall that the sandpile group $\Gamma(G)$ of a connected graph $G$ is the cokernel of the reduced laplacian matrix $\Delta'$.

Let $0<\alpha,q<1$ be constants. We define a random bipartite graph $G=G(n,\alpha,q)$ as follows: Take two sets of vertices $L$ and $R$ with $|L|=n,|R|=\lfloor\alpha n\rfloor$, and for each pair of vertices $v\in L$ and $u\in R$, include the edge between $v$ and $u$ independently with probability $q$.

We now state our main result about the $p$-rank of $\Gamma(G)$:

\begin{thm}
\label{thm:main}

Let $G=G(n,\alpha,q)$ be a random bipartite graph, and $p$ a prime. Then as $n\rightarrow \infty$, the expected value of the $p$-rank of the sandpile group $\Gamma(G)$ is:

\begin{enumerate}
\item $\left(\frac{1}{p}-\alpha\right)n+O(1)$ if $\alpha<\frac{1}{p}$
\item $O(1)$ if $\alpha>\frac{1}{p}$
\item $\sqrt{\frac{\frac{1}{p}(1-\frac{1}{p})n}{2\pi}}$+$O(1)$ if $\alpha=\frac{1}{p}.$
\end{enumerate}

\end{thm}

\noindent It is worth noting that the limits in the theorem do not depend on the value of $q$.

Theorem~\ref{thm:strongMain} will also give us explicit information about the distribution of the $p$-ranks. From numerical computations, it appears that the $O(1)$ constants in the first two cases of the theorem are at most $1$, and the $O(1)$ constant in the third case is around $2$.

The proof of Theorem~\ref{thm:main} relies on the assumption that $\alpha<1$. Based on numerical computations of random graphs, we conjecture that Theorem~\ref{thm:main} also holds when $\alpha=1$. This implies that the expected $p$-rank of the sandpile group of a balanced bipartite graphs should be $O(1)$ for all primes $p$. However, the best that can be done with our methods is:
\begin{cor}
\label{cor:balanced}
Let $G=G(n,1,q)$ be a random balanced bipartite graph, $p$ prime. Then as $n\rightarrow \infty$, the expected value of the $p$-rank of the sandpile group $\Gamma(G)$ is $o(n)$.
\end{cor}
\noindent Which we prove in Section~\ref{sec:details}.

\subsection{Connection to \ER Random Graphs}

It is interesting to ask what the distribution of the sandpile groups of random graphs looks like. The authors of \cite{pk} noted that the sandpile group of a graph comes with a canonical symmetric perfect bilinear pairing $\langle\cdot,\cdot\rangle_{G}$, and conjectured that for an \ER random graph $G$, the pair $(\Gamma(G),\langle\cdot,\cdot\rangle)$ of the sandpile group and its associated pairing can be predicted by certain heuristics of Cohen-Lenstra type.

The Cohen-Lenstra heuristics are an attempt to model what a generic ``random'' group should look like. In \cite{pk}, the authors show that the cokernel of a random symmetric matrix over $\Z_{p}$, distributed according to the Haar measure, follows heuristics of Cohen-Lenstra type, and conjectured that the sandpile groups of \ER random graphs should follow the same heuristics. 

In \cite{mw}, Melanie Wood proves several results in this direction. In particular, she shows that for an \ER random graph $G$, the $p$-part of $\Gamma(G)$ follows these heuristics for any finite collection of primes $p$.

However, Theorem~\ref{thm:main} shows that sufficiently unbalanced random bipartite graphs do not follow any similar type of Cohen-Lenstra heuristics: For example, the Cohen-Lenstra heuristics predict that for any $p$, the expected $p$-rank of $\Gamma(G)$ should stay low as $n$ grows. However, Theorem~\ref{thm:main} implies that for sufficiently unbalanced bipartite graphs, the $p$-rank grows linearly with $n$. Furthermore, the Cohen-Lenstra heuristics predict that the probability that $\Gamma(G)$ is cyclic should converge to a constant between $0$ and $1$, but in Section~\ref{sec:details} we prove that this is not the case for sufficiently unbalanced bipartite graphs.

\begin{cor}
\label{cor:cyclic}
Let $G=G(n,\alpha,q)$ be a random bipartite graph with $\alpha<\frac{1}{2}$. Then as $n\rightarrow\infty$, the probability that $\Gamma(G)$ is cyclic goes to zero exponentially fast.

\end{cor}

\noindent Because of the $O(1)$ factor in Theorem~\ref{thm:main}, the theorem gives us no information on the probability that $\Gamma(G)$ is cyclic when $\alpha\geq\frac{1}{2}$. Numerical computations suggest that this probability converges to a constant around $0.60$ when $\alpha>\frac{1}{2}$, and to a constant around $0.29$ when $\alpha=\frac{1}{2}$. 

Here is a brief outline of the paper: In Section~\ref{sec:closeness}, we define when sequences of random variables are ``usually within small distance'', which will give us a useful equivalence relation for random variables. We also give Theorem~\ref{thm:strongMain}, which describes the distribution of the $p$-rank of $S(\Gamma)$, and show that it implies Theorem~\ref{thm:main}.

In Section~\ref{sec:entropy}, we introduce our notion of min-entropy, which is a variant on the one used by Maples in \cite{km}. This notion is meant to replace independence; the matrices we will work with are not independent, but they are ``almost independent'', in the sense described by min-entropy, which will suffice for our purposes. 

In Section~\ref{sec:proof} we introduce a random matrix $M$, whose corank is usually within small distance of the $p$-rank of $\Gamma(G)$. Using the min-entropy properties of $M$, we will show that its corank is also usually within small distance of the distribution given in Theorem~\ref{thm:strongMain}, which will complete our proof.

Section~\ref{sec:background} contains some background information, and Section~\ref{sec:details} contains proof of the corollaries of Theorem~\ref{thm:main}.

\textbf{Acknowledgments}.
The author is grateful to Sam Payne and Nathan Kaplan for suggesting the problem, as well as their many helpful suggestions along the way. Also to Dan Carmon, for suggesting the proof of Claim~\ref{claim:estimation}.

This work was partially supported by NSF CAREER DMS-1149054. 

\section{The Sandpile Group, Binomial Distributions, and Schur Complements}
\label{sec:background}

\subsection{The Sandpile Group.}
In this section, we define the sandpile group. For a more thorough introduction to the subject with some lovely pictures, see \cite{wias}.

Let $G$ be a connected graph on $n$ vertices, numbered $1$ through $n$. The \textbf{laplacian matrix} of $G$ is the $n\times n$ matrix $\Delta=D-A$, where $A$ is the adjacency matrix of $G$ and $D$ is the diagonal degree matrix of $G$. In other words, $\Delta_{ii}=\deg(v_{i})$ and for $i\neq j$, $\Delta_{ij}=-1$ if $G$ has an edge between vertices $i$ and $j$, ans to $0$ otherwise. Note that $\Delta$ is a symmetric matrix whose rows and columns sum to zero, so it is singular. In fact, $\corank(\Delta)$ is equal to the number of connected components of $G$, where we define the corank of an $n\times m$ matrix $A$ as $\min(n,m)-\rank(A)$.

Choose a vertex $i$. The \textbf{reduced laplacian matrix} $\Delta'$ is the $(n-1)\times(n-1)$ matrix obtained by removing row $i$ and its corresponding column from $\Delta$. The \textbf{sandpile group} $\Gamma(G)$ is the cokernel of $\Delta'$, that is, $\Gamma(G)=\Z^{n-1}/\Delta'(\Z^{n-1})$.

It is shown in \cite{wias} that the sandpile group of a graph is independent of the choice of the vertex $i$. Moreover, the Matrix Tree Theorem shows that for a connected graph $G$, the determinant $\det(\Delta')$ is equal to the number of spanning trees of $G$. In particular, $\Delta'$ has full rank, so $\Z^{n-1}/\Delta'(\Z^{n-1})$ is a finite group of order $\det(\Delta')$ and rank at most $n-1$.

If $G$ is disconnected, we define its sandpile group to be the direct sum of the sandpile groups of its connected components. It is easy to see that this is a finite group of rank at most $n-1$. Moreover, it is shown in \cite{ra} that a random bipartite graph $G(n,\alpha,q)$ is connected with probability $1-O(e^{-Kn})$ for some $K>0$ depending only on $q$ and $\alpha$. This will allow us to consider the rank of the cokernel of the reduced laplacian rather than the rank of the sandpile group directly, as they are equal with probability $1-O(e^{-Kn})$.

\subsection{Binomial and Normal Distributions} 
We use $B(n,q)$ for the binomial distribution, the sum of $n$ independent Bernoulli random variables equal to $1$ with probability $q$ and $0$ otherwise. Recall that $\E(B(n,q))=qn$, where $\E(X)$ is the expected value of $X$.

We will make repeated use of Hoeffding's inequality:

\begin{thm}[Hoeffding's inequality]
\label{thm:hoeffding}
Let $B(n,q)$ be the binomial distribution, and let $\epsilon>0$. Then there exists a constant $K>0$, depending only on $q$ and $\epsilon$, such that $\P\left(\left|B(n,q)-qn\right|>\epsilon n\right)<e^{-Kn}$.
\end{thm}
For the proof, see for example \cite{hf}.

\subsection{Schur Complements}

Finally, we recall the basics of Schur complements, which will be a central tool in our proof. For a more thorough introduction to the subject, see \cite[Chapter 1]{zs}.

\begin{defn}
Let $A$ be an $n\times n$ matrix. Let $S$ be a subset of ${1,\dots,n}$, and let $T$ be the complement of $S$. We write $A_{S,S}$ for the submatrix given by restricting $A$ to the rows and columns whose indices are in $S$, $A_{T,T}$ for the submatrix of rows and columns with indices in $T$, and $A_{S,T}$ for the submatrix of rows in $A$ and columns in $T$.
\end{defn}

\noindent For example, if $S=\{1,\dots,k\}$, then $A=\begin{pmatrix}
A_{S,S} & A_{S,T}\\ 
A_{T,S} & A_{T,T}
\end{pmatrix} $

\begin{defn}
Let $A,S,T$ as above. If $A_{S,S}$ is invertible, then we define the \textbf{Schur complement} $A/S$ (or $A/A_{S,S}$) by $A/S=A_{T,T}-A_{T,S}A_{T,T}^{-1}A_{S,T}$.
\end{defn}

\noindent Note that $A/S$ is a $|T|\times |T|$ matrix.

Recall that the \textbf{corank} of an $n\times m$ matrix $A$ is defined as $\min(n,m)-\rank(A)$. We will use the following theorem several times:

\begin{thm}
Let $A$ and $S$ as above such that $A_{S,S}$ is invertible. Then $\corank(A/S)=\corank(A)$.
\end{thm}

\begin{proof}
Assume that $S$ is composed of the first $k$ entries for some $k\leq n$. It can be seen that
\[A=\begin{pmatrix}
A_{S,S} & A_{S,T}\\ 
A_{T,S} & A_{T,T}
\end{pmatrix} =\begin{pmatrix}
I_{k} & 0 \\ 
A_{T,S}A_{S,S}^{-1} & I_{n-k}
\end{pmatrix} \begin{pmatrix}
A_{S,S} & A_{S,T}\\ 
0 & A/S
\end{pmatrix}. \]

\noindent Where $I_k$ is the $k\times k$ identity matrix. Since the matrix on the left is invertible, we get that 
\[\corank(A)=\corank \begin{pmatrix}
A_{S,S} & A_{S,T}\\ 
0 & A/S
\end{pmatrix}. \]

As $A_{S,S}$ is invertible, row reduction gives us 
\[\corank \begin{pmatrix}
A_{S,S} & A_{S,T}\\ 
0 & A/S
\end{pmatrix} = \corank(A/S), \]

which completes the proof.
\end{proof}

\section{Closeness of Random Variables}
\label{sec:closeness}

In this section, we define when random variables are usually within small distance. This describes the ``closeness'' of random variables in a useful way.

\begin{defn}
\label{defn:closeness}
Let $X_{n},Y_{n}$ be two sequences of random variables. We say $X_{n}$ and $Y_{n}$ are \textbf{usually within small distance} if there exist constants $c,K>0$ such that for every $n,m>0$, $\P(|X_{n}-Y_{n}|\geq m)\leq Ke^{-cm}$.
\end{defn}

We will use the following properties.

\begin{lemma}
\label{lemma:closeness}

\begin{enumerate}
\item If $X_{n},Y_{n}$ and $Y_{n},Z_{n}$ are pairs of sequences of random variables which are usually within small distance, then so are $X_{n},Z_{n}$. Hence being usually within small distance is an equivalence relation for sequences of random variables.
\item If $X_{n},Y_{n}$ are sequences of random variables which usually within small distance, then $\left|\E(X_{n})-\E(Y_{n})\right|=O(1)$.
\item If $\P(X_{n}\neq Y_{n})=1-O(e^{-cn})$ for some constant $c>0$ and $X_{n},Y_{n}$ are bounded by $O(n)$, then $X_n,Y_n$ are usually within small distance.
\item If $X_{n},Y_{n}$ are sequences of random variables and $\max|X_{n}-Y_{n}|<K$ for some constant $K$, then $X_{n},Y_{n}$ are usually within small distance.
\item If $X_{n},Y_{n},Z_{n}$ are sequences of random variables such that $X_{n}\leq Y_{n}\leq Z_{n}$ and $Z_{n}$ is usually within small distance of $X_{n}$, then so is $Y_{n}$.
\item If $X_{n},Y_{n}$ are usually within small distance of $X'_{n},Y'_{n}$ respectively, and $C$ is constant, then $X_{n}+Y_{n},CX_{n},\max(X_{n},Y_{n})$, and $\min(X_{n},Y_{n})$ are usually within small distance of $X'_{n}+Y'_{n},CX'_{n},\max(X'_{n},Y'_{n})$, and $\min(X'_{n},Y'_{n})$ respectively.
\end{enumerate}
\end{lemma}
\noindent The proofs are straightforward.

Using the above definition, we can now state the main theorem about the distribution of the $p$-rank of the sandpile group, from which we will deduce Theorem~\ref{thm:main}:

\begin{thm}
\label{thm:strongMain}
Let $G=G(n,\alpha,q)$ be a random bipartite graph, and $p$ a prime. Let $X_{n}=X(n,\alpha,q,p)$ be the $p$-rank of $\Gamma(G)$, and recall that $B(n,q)$ denotes a binomial random variable. Then $X_{n}$ is usually within small distance of $\max\left(B\left(n,\frac{1}{p}\right)-\alpha n,0\right)$, where $B\left(n,\frac{1}{p}\right)$ is the binomial distribution.
\end{thm}

We will prove Theorem~\ref{thm:strongMain} in Section~\ref{sec:proof}. First, we show:

\begin{prop}
\label{prop:reduction}
Theorem~\ref{thm:strongMain} implies Theorem~\ref{thm:main}.
\end{prop}

\begin{proof}

By Lemma~\ref{lemma:closeness}, Theorem~\ref{thm:strongMain} implies that 
\[\E(X_{n})=\E\left(\max\left(B\left(n,\frac{1}{p}\right)-\alpha n,0\right)\right)+O(1).\]
Hence it suffices to calculate $\E\left(\max\left(B\left(n,\frac{1}{p}\right)-\alpha n,0\right)\right)$. We will split into three cases, depending on whether $\alpha<\frac{1}{p},\alpha>\frac{1}{p}$, or $\alpha=\frac{1}{p}$.

\vspace{12 pt}
\noindent{\bf The case $\alpha < \frac{1}{p}.$}
\vspace{12 pt}

Note that by Hoeffding's inequality $\P\left(B\left(n,\frac{1}{p}\right)-\alpha n>0\right)=1-O(e^{-cn}) $ for some constant $c>0$. Hence by lemma~\ref{lemma:closeness}, $B\left(n,\frac{1}{p}\right)-\alpha n$ is usually within small distance of $\max\left(B\left(n,\frac{1}{p}\right)-\alpha n,0\right)$. 

Because of this, it suffices to calculate $\E\left(B\left(n,\frac{1}{p}\right)-\alpha n\right)$. Using the additivity of the expected value, we see that
\[\E\left(B\left(n,\frac{1}{p}\right)-\alpha n\right)=\E\left(B\left(n,\frac{1}{p}\right)\right)-\alpha n = \frac{1}{p}n-\alpha n=\left(\frac{1}{p}-\alpha\right)n.\]

\vspace{12 pt}
\noindent{\bf The case $\alpha > \frac{1}{p}.$}
\vspace{12 pt}

This case is similar. Again by Hoeffding's inequality, we get that $\P\left(B\left(n,\frac{1}{p}\right)-\alpha n>0\right)=O(e^{-cn})$ and hence $\max\left(B\left(n,\frac{1}{p}\right)-\alpha n,0\right)$ is equal to $0$ with probability $1-O(e^{-cn})$. Hence $\max\left(B\left(n,\frac{1}{p}\right)-\alpha n,0\right)$ is usually within small distance of $0$, which has expected value $O(1)$.

\vspace{12 pt}
\noindent{\bf The case $\alpha = \frac{1}{p}.$}
\vspace{12 pt}

Finally, the case where $\alpha =\frac{1}{p}$. In this case, we wish to calculate $\E\left(\max\left(B\left(n,\alpha\right)-\alpha n,0\right)\right)$. We will rely on the following claim:
\begin{claim}
\label{claim:estimation}
Let $B(n,\alpha)$ be the binomial distribution, $s$ a positive integer. Then 
\[\E(B(n,\alpha)|B(n,\alpha)>s)=\alpha n+\alpha(1-\alpha)n\frac{\P(B(n-1,\alpha)=s)}{\P(B(n,\alpha)>s)}.\]
\end{claim}

\begin{proof}
Let $Y=B(n,\alpha)$. We wish to calculate $\E(Y|Y>s)=\frac{\sum_{k>s}k\P(Y=k)}{\sum_{k>s}\P(Y=k)}$. Since we expect the main term in the expectation to be $\E\left(B(n,\alpha)\right)= \alpha n$, we wish to estimate
\[\E(Y|Y>s)-\alpha n=\frac{\sum_{k>s}k\P(Y=k)}{\sum_{k>s}\P(Y=k)}-\alpha n = \frac{\sum_{k>s}k\P(Y=k)-\sum_{k>s} \alpha n\P(Y=k)}{\sum_{k>s}\P(Y=k)}.\]

Now consider the two sums 
\begin{align}
& \sum_{k>s} k\P(Y=k) =\sum_{k>s} k\alpha^k(1-\alpha)^{n-k} \binom{n}{k} \\
& \sum_{k>s} \alpha n\P(Y=k) = \alpha n \sum_{k>s} \alpha^k(1-\alpha)^{n-k} \binom{n}{k}.
\end{align}

We manipulate the sums as follows: In sum (1), replace $k\binom{n}{k}$ with the equal $n\binom{n-1}{k-1}$ and take $\alpha n$ out, so that it becomes $\alpha n\sum_{k>s}\alpha^{k-1}(1-\alpha)^{n-k}\binom{n-1}{k-1}$.

Now, multiply by $\alpha+(1-\alpha)=1$, and expand, to obtain the two sums $(1)=(1a)+(1b)$, where
\begin{align*}
&(1a)\  &&\alpha n\sum_{k>s} \alpha^{k}(1-\alpha)^{n-k}\binom{n-1}{k-1} \\
&(1b)\  &&\alpha n\sum_{k>s} \alpha^{k-1}(1-\alpha)^{n-k+1}\binom{n-1}{k-1}.
\end{align*}

For sum (2), use $\binom{n}{k}=\binom{n-1}{k}+\binom{n-1}{k-1}$ to obtain $(2)=(2a)+(2b)$, where
\begin{align*}
&(2a)\  &&\alpha n\sum_{k>s} \alpha^{k}(1-\alpha)^{n-k}\binom{n-1}{k} \\
&(2b)\  &&\alpha n\sum_{k>s} \alpha^{k}(1-\alpha)^{n-k}\binom{n-1}{k-1}.
\end{align*}

Now the difference $(1)-(2)$ cancels out! Observe that $(1a)=(2b)$, whereas
$(1b)$ and $(2a)$ are just shifts of each other, so the difference cancels out in a telescopic sum, and we obtain
\[(1)-(2)=(1b)-(2a)=\alpha n \left(\alpha^{s}(1-\alpha)^{n-s}\binom{n-1}{s}\right)=n\alpha (1-\alpha)\P(B(n-1,\alpha)=s). \]

Finally, putting our expression for $(1)-(2)$ back in our equation for the expectation, we get 

\begin{align*}
\E(Y|Y>s)-\alpha n &= \frac{\sum_{k>s}k\P(Y=k)-\sum_{k>s} \alpha n\P(Y=k)}{\sum_{k>s}\P(Y=k)} \\
&= \frac{n\alpha (1-\alpha)\P(B(n-1,\alpha)=s)}{\P(Y>s)},
\end{align*}
\noindent which completes the proof of the claim.
\end{proof}

For estimating $\E\left(\max\left(B\left(n,\alpha\right)-\alpha n,0\right)\right)$, the following version of the De Moivre-Laplace theorem will be useful.
\begin{thm}[{\cite[Theorem 2]{mt}}]
\label{thm:dml}
Let $s$ be an integer such that $|\alpha n-s|<\sqrt{n}$. Then 
\[\P(B(n,\alpha)=s)=\frac{1}{\sqrt{2\pi \alpha (1-\alpha)n}} e^{-\frac{(s-\alpha n)^{2}}{2\alpha(1-\alpha)n}}\left(1+O\left(\frac{1}{\sqrt{n}}\right)\right).\]
\end{thm}

In our calculation, we will need to estimate $\P(B(n-1,\alpha)=s)$ for $s=\lfloor\alpha n\rfloor$. As $(s-\alpha n)^{2}\leq 1$, we get that 
\[\left|e^{-\frac{(s-\alpha (n-1))^{2}}{2\alpha(1-\alpha)(n-1)}}-1\right|\leq\left|e^{-\frac{1}{2\alpha(1-\alpha)(n-1)}}-1\right|=O\left(\frac{1}{n}\right),\]
and hence $ e^{-\frac{(s-\alpha (n-1))^{2}}{2\alpha(1-\alpha)(n-1)}}=1+O\left(\frac{1}{n}\right)$. As $\frac{1}{\sqrt{n-1}}=\frac{1}{\sqrt{n}}\left(1+O\left(\frac{1}{\sqrt{n}}\right)\right)$ we have by Theorem~\ref{thm:dml}: 

\begin{align}
\nonumber \P(B(n-1,\alpha)=s)&=\frac{1}{\sqrt{2\pi \alpha (1-\alpha)(n-1)}} e^{-\frac{(s-\alpha (n-1))^{2}}{2\alpha(1-\alpha)(n-1)}}\left(1+O\left(\frac{1}{\sqrt{n}}\right)\right) \\ \nonumber
&=\frac{1}{\sqrt{2\pi \alpha (1-\alpha)n}} \left(1+O\left(\frac{1}{\sqrt{n}}\right)\right) \left( 1+O\left(\frac{1}{n}\right) \right) \left(1+O\left(\frac{1}{\sqrt{n}}\right)\right)\\ 
&=\frac{1}{\sqrt{2\pi \alpha (1-\alpha)n}} \left(1+O\left(\frac{1}{\sqrt{n}}\right)\right). \label{eqn:test}
\end{align}

Recall that we wish to estimate $\E\left(\max\left(B\left(n,\alpha\right)-\alpha n,0\right)\right)$. 

\begin{align*}
\E(\max(B(n,\alpha)-\alpha n,0)&=\P(B(n,\alpha)>\lfloor\alpha n\rfloor)\E(B(n,\alpha)-\alpha n|B(n,\alpha)>\lfloor\alpha n\rfloor)\\
&=\P(B(n,\alpha)>\lfloor\alpha n\rfloor)(\E(B(n,\alpha)|B(n,\alpha)>\lfloor\alpha n\rfloor)-\alpha n).
\end{align*}
Using Claim~\ref{claim:estimation} with $s=\lfloor\alpha n\rfloor$, we get:
\begin{align*}
\P(B(n,\alpha)>\lfloor\alpha n\rfloor)&(\E(B(n,\alpha)|B(n,\alpha)>\lfloor\alpha n\rfloor)-\alpha n) \\
&=\P(B(n,\alpha)>\lfloor\alpha n\rfloor)\left(\alpha n+\alpha(1-\alpha)n\frac{\P(B(n-1,\alpha)=\lfloor\alpha n\rfloor)}{\P(B(n,\alpha)>\lfloor\alpha n\rfloor)}-\alpha n\right)\\
&=\alpha(1-\alpha)n\P(B(n-1,\alpha)=\lfloor\alpha n\rfloor).
\end{align*}
Finally, using \ref{eqn:test}, we get:
\begin{align*}
\alpha(1-\alpha)n\P(B(n-1,\alpha)=\lfloor\alpha n\rfloor) &=\alpha(1-\alpha)n\frac{1}{\sqrt{2\pi \alpha (1-\alpha)n}} \left(1+O\left(\frac{1}{\sqrt{n}}\right)\right)\\
&=\sqrt{\frac{\alpha(1-\alpha)n}{2\pi}}\left(1+O\left(\frac{1}{\sqrt{n}}\right)\right)\\
&=\sqrt{\frac{\alpha(1-\alpha)n}{2\pi}}+O(1).
\end{align*}

and substituting $\alpha=\frac{1}{p}$ gives us the expression from Theorem~\ref{thm:main}.
\end{proof}

\section{Min-Entropy and Random Matrix Rank}
\label{sec:entropy}

In this section, we define our notion of min-entropy, which is a variant on the definition given by Maples in \cite{km} and use it to prove some lemmas which will be useful in the proof of Theorem~\ref{thm:strongMain}.

\begin{defn}
Let $A$ be a random matrix over $\Z/p\Z$. Let $\beta>0$, and let $I$ be a set of entries in $A$. We say that \textbf{an entry $A_{i_{0}j_{0}}\in A$ has min-entropy at least $\beta$ with respect to $I$} if, for any choice of values $a_{ij}$ for the entries in $I$ that can occur with nonzero probability, and every $a\in \Z/p\Z$, the probability $\P(A_{i_{0}j_{0}}=a|A_{ij}=a_{ij}\forall (i,j)\in I)$ is at most $1-\beta$.

We say that \textbf{the matrix $A$ has min-entropy at least $\beta$} if every entry of $A$ has min-entropy at least $\beta$ with respect to the set of all other entries. 
\end{defn}

In other words, $A_{ij}$ has min-entropy greater than $\beta$ relative to a set of entries if fixing them cannot control $A_{ij}$, in the sense that it still has probability at most $1-\beta$ of being any specific value.  We can think of min-entropy as a bound on how much fixing some entries of a matrix can influence other entries. We illustrate this notion of min-entropy with the following examples.

If all the entries of $A$ are independent, the min-entropy of $A_{ij}$ is simply $\min_{x}(1-\mathbb{P}(A_{ij}=x))$. In particular, if the entries of $A_{ij}$ are all independent and uniformly distributed in $\mathbb{Z}/p\Z$, this min-entropy is $1-\frac{1}{p}$, which is the highest possible.

For another example, consider $\Delta=\Delta(n,\alpha,q)$, the laplacian matrix of a random bipartite graph. Since every row in $\Delta$ sums to zero, for any entry $\Delta_{ij}$, fixing the rest of the entries in row $i$ determines $\Delta_{ij}$. Hence $\Delta_{ij}$ has zero min-entropy with respect to the rest of the entries in row $i$.

\begin{thm}
\label{thm:rectrank}
Let $A$ be an $n\times m$ random matrix over $\Z/p\Z$, for $m\geq n$, with min-entropy at least $\beta$ for some $\beta>0$. Then the probability that $A$ has rank $n$ is at least $1-\frac{1}{\beta^{2}}(1-\beta)^{m+1-n}$. In particular, there exists a constant $K>0$ depending only on $\beta$ such that $\P(\rank(A)=n)\geq 1-e^{-K(m-n)}$.
\end{thm}

\begin{proof}
Let $v_{1},\dots,v_{n}$ be the rows of $A$. Then $A$ has rank $n$ only if the $v_{i}$ are independent, so the probability $\P(\rank(A)=n)$ is equal to the product 
\[\prod_{i=1}^{n}\mathbb{P}(v_{i}\text{ is independent of }\{v_{1},\dots,v_{i-1}\}|\{v_{1},\dots,v_{i-1}\} \text{ are independent}).\]

We now note that for each $i$, 
\[\P(v_{i}\text{ is independent of }\{v_{1},\dots,v_{i-1}\}|\{v_{1},\dots,v_{i-1}\} \text{ are independent})\geq (1-\beta)^{m-(i-1)}. \]
To see this, assume that $\{v_{1},\dots,v_{i-1}\}$ are independent. Then there exists a subset $J=\{j_{1},\dots,j_{i-1}\}\subset [m]$ such that the restrictions of the $\{v_{j}\}_{j<i}$ to the entries in $J$ are independent. 

Assume that $J=\{1,\dots,i-1\}$. by the independence of the $v_{l}|_{J}$, there exist unique coefficients $a_{1},\dots,a_{i-1}$ such that for all $j<i$, $(v_{i})_{j}=\sum_{l<i} a_{l}(v_{l})_{j}$. 

$v_{i}$ is dependent on $\{v_{1},\dots,v_{i-1}\}$ only if there exists a linear combination of them that sums to $v_{i}$. By the uniqueness of the coefficients $a_{1},\dots,a_{i-1}$, this happens only if $v_{i}=\sum_{l<i}a_{l}v_{l}$. In particular, $v_{i}$ is dependent on the previous row vectors only if for all $j\geq i$, $(v_{i})_{j}=\sum_{l<i} a_{l}(v_{l})_{j}$.

However, by the min-entropy assumption, this happens for each $j$ with probability at most $1-\beta$. As there are $m-(i-1)$ such entries, the probability that this equality holds for all of them is at most $(1-\beta)^{m-(i-1)}$. Hence the probability that $v_{i}$ is independent of $\{v_{1},\dots,v_{i-1}\}$ is at least $1-(1-\beta)^{m-(i-1)}$. Using this, we get the following bound 

\begin{align}
\prod_{i=1}^{n}\mathbb{P}(v_{i}\text{ is independent of }\{v_{1},\dots,v_{i-1}\}|\{v_{1},\dots,v_{i-1}\} \text{ are independent}) \nonumber \\
\geq \prod_{i=1}^{n}(1-(1-\beta)^{m-(i-1)})=(1-(1-\beta)^{m})\cdots(1-(1-\beta)^{m+1-n}).\label{eqn:betaprod}
\end{align}

We now wish to bound (\ref{eqn:betaprod}) from below.
\begin{claim}
The product $(1-(1-\beta)^{m})\cdots(1-(1-\beta)^{m+1-n})$ is at least $1-\frac{1}{\beta^{2}}(1-\beta)^{m+1-n}$.
\end{claim}

Write $\gamma=1-\beta,r=m+1-n$. We need to find a lower bound on the product $(1-\gamma^{m})\cdots(1-\gamma^{r})$. We will rely on the fact that $0<\gamma<1$.

First, recall that for any positive $x$, $x\geq\log(x)+1$. Using this for $x=(1-\gamma^{m})\cdots(1-\gamma^{r})$, we get:
\[(1-\gamma^{m})\cdots(1-\gamma^{r})\geq 1+\log\left((1-\gamma^{m})\cdots(1-\gamma^{r})\right)\]
Now split the product to get:
\[1+\log((1-\gamma^{m})\cdots(1-\gamma^{r}))=1+\sum_{i=r}^{m}\log(1-\gamma^{i})\]

For any $0<x<1$, $\log(1-x)>-\frac{x}{1-x}$. To see this, let $h(t)=\frac{1}{1-x}(t-(1-x))+\log(1-x)$ be the tangent line to $\log(t)$ at $t=1-x$. Then as $\log(t)$ is concave, $h(1)=\frac{x}{1-x}+\log(1-x)>\log(1)=0$, so $\log(1-x)>-\frac{x}{1-x}$. Using this for $x=1-\gamma^{i}$, we get:
\[1+\sum_{i=r}^{m}\log(1-\gamma^{i})\geq 1+\sum_{i=r}^{m}\frac{-\gamma^{i}}{1-\gamma^{i}}\]
As $\gamma<1$, we have:
\[1+\sum_{i=r}^{m}-\frac{\gamma^{i}}{1-\gamma^{i}}\geq 1+\sum_{i=r}^{m}\frac{-\gamma^{i}}{1-\gamma}=1-\frac{1}{1-\gamma}\sum_{i=r}^{m}\gamma^{i}\]

We will now bound this by the sum of the infinite series:
\[1-\frac{1}{1-\gamma}\sum_{i=r}^{m}\gamma^{i}\geq1-\frac{1}{1-\gamma}\sum_{i=r}^{\infty}\gamma^{i}=1-\frac{\gamma^{r}}{(1-\gamma)^{2}}\].

Translating back through $\gamma=1-\beta$,$r=m+1-n$, this is $1-\frac{1}{\beta^{2}}(1-\beta)^{m+1-n}$, which proves the claim, and the theorem follows.
\end{proof}

\begin{cor}
\label{cor:rectrank}
Let $A_{n}$ be an $n\times m$ random matrix over $\Z/p\Z$ with min-entropy at least $\beta$ for some $\beta>0$ independent of $n$. Then $\corank(A)$ is usually within small distance of $0$.
\end{cor}

\begin{proof}
Let $s>0$, and assume that $m\geq n$. We wish to show that $\P(\corank(A)>s)=O(e^{-Ks})$, where $K>0$ is independent of $n$.
Let $A'$ be the submatrix of $A$ given by taking the first $n-s$ rows. Then $A'$ is an $n-s\times m$ matrix, so by Theorem~\ref{thm:rectrank}, its rows are independent with probability a probability at least $1-e^{-K(m-(n-s))}=1-(e^{-K(m-n)})e^{-Ks}\geq 1-e^{-Ks}$, where $K$ depends only on $\beta$. But if $A'$ has rank $n-s$, the corank of $A$ is at most $s$, so $\P(\corank(A)>s)\leq e^{-Ks}$.
\end{proof}

\section{Proof of Theorem~\ref{thm:strongMain}}
\label{sec:proof}

For this section, we fix a prime $p$, as well as constants $0<\alpha,q<1$.

We will now prove Theorem~\ref{thm:strongMain}. We do this in two stages. First, we reduce the laplacian mod $p$, remove the first and last $p$ rows and columns, and set the diagonal entries to be uniformly distributed mod $p$. We call the resulting matrix $M$. We show that $\corank(M)$ is usually within small distance of the $p$-rank of the sandpile group, which reduces Theorem~\ref{thm:strongMain} to calculating the distribution of $\corank(M)$.

In the second stage, we calculate the distribution of $\corank(M)$. Removing some of the rows and columns of the laplacian will allow the upper triangular entries of $M$ to have positive min-entropy with respect to the other upper triangular matrix, which will allow us to use Corollary~\ref{cor:rectrank} to compute $\corank(M)$. 

\subsection{Reduction to $M$}

Let $(L,R)$ be the vertices of our random bipartite graph $G$, and let $\Delta$ be the laplacian of $G$. Note that $\Delta$ is of the form $\bigl(\begin{smallmatrix}
D_{0,1} & -A_{0} \\ 
-A_{0}^{T} & D_{0,2}
\end{smallmatrix}\bigr)$, where $A_{0}$ is the adjacency matrix between $L$ and $R$ and $D_{0,1}$ and $D_{0,2}$ are diagonal matrices. Since we wish to work over $\Z/p\Z$, we will consider $\Delta\otimes \Z/p\Z=\Delta/p$.

As we saw earlier, $\Delta/p$ has min-entropy $0$. We resolve this issue by using the submatrix $\Delta_{1}$, which has positive min-entropy.
\begin{defn}
Let $G$ be a bipartite graph with laplacian $\Delta$, $p$ prime. We define the matrix $\Delta_{1}=\Delta_{1}(G)=\Delta_{1}(n,\alpha,q,p)$ over $\Z/p\Z$ to be the submatrix of $\Delta/p$ given by removing the first $p$ rows, the first $p$ columns, the last $p$ rows, and the last $p$ columns.
\end{defn}

\begin{lemma}
\label{lemma:properties}
Let $G=G(n,\alpha,q)$ be as above, and let $\Delta_{1}=\Delta_{1}(G)$. Write $\Delta_{1}=\bigl(\begin{smallmatrix}
D_{1,1} & -A_{1} \\ 
-A_{1}^{T} & D_{1,2}
\end{smallmatrix}\bigr)$. $\Delta_{1}$ has the following properties:
\begin{enumerate}
\item The diagonal values of $D_{1,1}$ are independent of each other, as well as of entries of $A_{1}$ outside of their row.
\item The diagonal values of $D_{1,2}$ are independent of each other, as well as of entries of $A_{1}$ outside of their column.
\item There exists $\beta>0$ depending only on $p,q,$ and $\alpha$ such that every non-constant entry in or above the diagonal in $\Delta_{1}$ has min-entropy at least $\beta$ with respect to the set of the entries in or above the diagonal.
\item For any $a\in\Z/p\Z$, and any diagonal entry $x$ in $D_{1,1}$ or $D_{1,2}$, $\P(x=a)= \frac{1}{p}+O(e^{-cn})$ for some constant $c$.
\end{enumerate}
\end{lemma}

\begin{proof}

We first show $(1)$. Note that the value of the diagonal entry $(D_{1,1})_{ii}$ depends only on the $i$\textsuperscript{th} row of $A_{0}$. Hence the $(D_{1,1})_{ii}$ are independent of each other and of any entry outside of the $i$\textsuperscript{th} row of $A_{0}$, which in particular includes the entries of $\Delta_{1}$ outside the $i$\textsuperscript{th} row. The proof of $(2)$ is similar.

We will now prove $(3)$. 

Let $x$ be an entry in the upper triangle of $\Delta_{1}$. If $x\in D_{1,1}$, then as $x$ is non-constant, it must be on the diagonal. As we saw above, $x$ depends only on the values in the $i$\textsuperscript{th} row of $A_{0}$. 

Fix the rest of the entries of the $i$\textsuperscript{th} row of $\Delta_{1}$. There are still $2p$ entries of the $i$\textsuperscript{th} row of $A_{0}$ not in $\Delta_{1}$, which are left undetermined. For any choice of the first $2p-1$ of these, the last entry can be either $-1$ with probability $q$ or $0$ otherwise, which would change the value of $x$. Hence $x$ has min-entropy at least $\min(q,1-q)$ with respect to the rest of the upper triangular entries. The case where $x\in D_{1,2}$ is similar.

Now, assume $x\in -A_{1}$. Fix all the other entries of $\Delta_{1}$. The only ones of which $x$ is not independent are those in the row and column of $x$. The row sum (in $A_{0}$) must be equal to the corresponding row entry, and the column sum must be equal to the corresponding column entry. 

There are $p$ unfixed entries in the row that are in $A_{0}$ but not in $\Delta_{1}$, and the sum of these entries can be equal to any value in $\Z/p\Z$ with probability at least $\min(q,1-q)^{p+1}$.The same goes for the column sum. In particular, the probability that both the row and the column sum allow $x$ to be zero is at least $\min(q,1-q)^{2(p+1)}$. Similarly, the probability that both allow $x=-1$ is at least $\min(q,1-q)^{2(p+1)}$. Hence $x$ has min-entropy at least $\min(q,1-q)^{2(p+1)}$ with respect to the rest of the upper triangular entries.

Finally, we prove $(4)$. Let $a\in\Z/p\Z$. To see that each entry of $D_{1,1}$ is equal to $a$ with probability $\frac{1}{p}+O(e^{-cn})$, note that it is equal to $a$ when the sum of the corresponding row in $A_{0}$ is equal to $a$. Since this row has $\alpha n$ independent entries equal to $1$ with probability $q$ and zero otherwise, its sum is uniformly distributed in $\Z/p\Z$ up to an $O(e^{-cn})$ error term, where $c$ is a constant depending only on $q,\alpha$ and $p$. 
\end{proof}

We will take $M$ to be equal to $\Delta_{1}(n+2p,\alpha,q,p)$, then adjust the probability space so that the diagonal values of $M$ are equidistributed in $Z/p\Z$. Since this changes only an exponentially small part of the probability space, $\corank(M)$ is usually within small distance of $\corank(\Delta_{1}(n+2p,\alpha,q,p))$. But 
\[\left|\corank(\Delta_{1}(n+2p,\alpha,q,p))-\corank(\Delta_{1}(n,\alpha,q,p))\right|<4p,\] 
so by transitivity $\corank(M)$ is usually within small distance of $\corank(\Delta_{1}(n,\alpha,q,p))$. Hence we have:

\begin{prop}
\label{prop:mtomain}
The $p$-rank of $\Gamma(G)$ is usually within small distance of $\corank(M)$.
\end{prop}

We will also assume that $\lfloor\alpha n\rfloor=\alpha n$, so that $M$ is a $(1+\alpha)n\times (1+\alpha)n$ matrix. We will write:

\[M=\begin{pmatrix}
D_{1} & A\\ 
A^{T} & D_{2}
\end{pmatrix}\]

\subsection{Calculating the corank of $M$}

In this section, we prove the following statement about $M$:
\begin{prop}
\label{prop:mrank}
Let $M=M(n,\alpha,q,p)=\begin{pmatrix}
D_{1} & A\\ 
A^{T} & D_{2}
\end{pmatrix}$ be the matrix described above. Then $\corank(M)$ is usually within small distance of $\max\left(B\left(n,\frac{1}{p}\right)-\alpha n,0\right)$.
\end{prop}

\noindent Together with Proposition~\ref{prop:mtomain}, this implies Theorem~\ref{thm:strongMain}.

Throughout the proof, we will use $\height(A)$ and $\width(A)$ to denote the number of rows and columns of $A$ respectively. If $A$ is a square matrix, we use $\dim(A)$ for both of these.

\begin{proof}

Let $r$ be the number of zero entries on the diagonal of $D_{1}$, that is, $r=\corank(D_{1})$. Since the diagonal values of $D_{1}$ are independent and uniformly distributed, it is easy to see that $r=B\left(n,\frac{1}{p}\right)$. Hence, it suffices to show that $\corank(M)$ is usually within small distance of $\max\left(r-\alpha n,0\right)$.

Our proof will rely on finding nonsingular submatrices of $M$, and taking the Schur complement with respect to them. This will allow us to reduce the problem of finding $\corank(M)$ to finding the coranks of matrices which are either nonsingular (in the case where $r-\alpha n< 0$), or have a large block of zeros which makes finding the corank straightforwards (in the case where $r-\alpha n\geq 0$). 

Assume that the first $n-r$ entries of $D_{1}$ are the nonzero entries, so that $D_{1}$ is of the form 
$\begin{pmatrix}
D_{1}' & 0\\ 
0 & 0
\end{pmatrix}$
Where $D_{1}'$ is invertible. Hence we can write
\[M=\begin{pmatrix}
D_{1}' &0  &B_{1} \\ 
0 &0  &B_{2} \\ 
B_{1}^{T} &B_{2}^{T}  &D_{2} 
\end{pmatrix} \]
where $B_{1},B_{2}$ are random matrices of dimension $\alpha n\times (n-r)$ and $\alpha n\times r$ respectively. Taking the Schur Complement of $M$ with respect to $D_{1}'$, we get:
\[M/D_{1}'=\begin{pmatrix}
0 & B_{2}\\ 
B_{2}^{T} & D_{2}-B_{1}^{T}D_{1}'^{-1}B_{1}
\end{pmatrix}.
\]
We will now split into cases:

\vspace{12 pt}
\noindent{\bf The case $r \ge \alpha n.$}
\vspace{12 pt}

In this case, we want to show that $\corank(M)$ is usually within small distance of $r-\alpha n$. Now, 
\[\height(B_{2})=r\geq\alpha n=\width(B_{2}).\]
As $\rank(B_{2})=\rank(B_{2}^{T})$, it is easy to see that 
\[\rank(M/D_{1}')\geq \rank(B_{2})+\rank(B_{2}^{T})=2\rank(B_{2}),\] 
and thus 
\[\corank(M)=\corank(M/D_{1}')\leq \dim(M/D_{1}')-\rank(M/D_{1}')=(\alpha n+r)-(2\rank(B_{2})).\]

Conversely, The corank of $M/D_{1}'$ is at least the corank of the submatrix of the top $n-r$ rows, given by $\begin{pmatrix}
0 & B_{2}
\end{pmatrix}. $ The rank of this submatrix is equal to $\rank(B_{2})$, so the corank is $r-\rank(B_{2})$.

Since $B_{2}$ has min-entropy at least $\beta$ for some positive constant $\beta$, by Corollary~\ref{cor:rectrank}, $\rank(B_{2})$ is usually within small distance of 
\[\min(\height(B_{2}),\width(B_{2}))=\min(r,\alpha n)=\alpha n.\]

Applying this to our lower and upper bounds for $\corank(M)$, we get that the upper bound is usually within small distance of $(\alpha n+r)-(2\alpha n)=r-\alpha n$. Similarly, our lower bound is usually within small distance of $r-\alpha n$. Hence by Lemma~\ref{lemma:properties}, $\rank(M)$ is usually within small distance of $r-\alpha n$.

\vspace{12 pt}
\noindent{\bf The case $r <\alpha n.$}
\vspace{12 pt}

In this case, we need to show that $\corank(M)=\corank(M/D_{1}')$ is usually within small distance of zero.

Write $C=D_{2}-B_{1}^{T}D_{1}'^{-1}B_{1}$ for the bottom-right $\alpha n\times \alpha n$ submatrix of $M/D_{1}'$. We will use the following claim:

\begin{claim}
Let $s$ be the size of the largest set of indices $J\subseteq \{1,\dots,\alpha n\}$ with the property that $C_{J}=(C_{ij|i,j\in J})$ is nonsingular. Then for any constant $\epsilon>0$, $s\geq \left(\alpha(1 -\frac{1}{p})-\epsilon\right) n$ with probability $1-O(e^{-cn})$ for some constant $c>0$.
\end{claim}

\begin{proof}
To see this, we build up a set $J$ by going through the indices $i\in\{1,\dots,\alpha n\}$. For each $i$, we add $i$ to $J$ if $C_{J\cup\{i\}}$ is nonsingular. We will show that for each $i$, we add $i$ with probability at least $1-\frac{1}{p}-\delta$, where $\delta>0$ is arbitrarily small as $n$ grows. Since $J$ is the sum of $n$ Bernoulli random variables, each equal to $1$ with probability at least $1-\frac{1}{p}-\delta$ independently of the previous values, we can say that $|J|\geq B\left(\alpha n, 1-\frac{1}{p}-\delta\right)$. 

By Hoeffding's inequality, 
\[B\left(\alpha n, 1-\frac{1}{p}-\delta\right)>(1-\delta)\alpha\left(1-\frac{1}{p}-\delta\right)n> \alpha\left(1 -\frac{1}{p}-\epsilon\right) n\] 
with probability $1-O(e^{-cn})$ (the second inequality holds for all sufficiently small $\delta$).

To see that each $i$ can be added with probability at least $1-\frac{1}{p}-\delta$, note that the diagonal entries of $D_{2}$ are  the sums of entries in $B_{1}$ with entries of $B_{2}$, which are independent of them. There are $r=B\left(n,\frac{1}{p}\right)$ entries in each column of $B_{2}$, so by Hoeffding's inequality the number of entries in each column of $B_{2}$ is greater than $\frac{1}{2p}n$ with probability $1-O(e^{-cn})$. Hence we can assume that the entries of $D_{2}$ are exponentially close to being uniformly distributed in $\Z/p\Z$, given any condition on $B_{1},D_{1}'$, and the previous diagonal entries of $D_{2}$. This means that for any $x\in Z/p\Z$ and any conditions on the rest of the entries of $C$, $\P(C_{ii}=x)\leq \frac{1}{p}+\delta$, where $\delta$ can be exponentially small in $n$.

Let $J$ be the set of indices we obtain from taking the above process on $\{1,\dots,i-1\}$. We need to show that we add $i$ to $J$ with probability at least $1-\delta-\frac{1}{p}$. We add $i$ to $J$ unless $C_{J\cup\{i\}}$ becomes singular. But this happens only if the last column of $C_{J\cup\{i\}}$ is dependent on the first $|J|$ columns.

Write $C_{J\cup\{i\}}=\begin{pmatrix}
u_{1}\cdots u_{|J|} & u_{i}\\ 
C_{ij_{0}}\dots C_{ij_{|J|}} & C_{ii}
\end{pmatrix}$. Since $C_{J}$ is nonsingular, there exist unique coefficients $a_{j}\in \Z/p\Z$ such that $\sum a_{j}u_{j}=u_{i}$. But $C_{J\cup\{i\}}$ only if its columns are dependent, which happens only if $\sum a_{j}C_{ij}=C_{ii}$. From the above statement with $x=\sum a_{j}C_{ij}$, this happens with probability at most $1-\delta-\frac{1}{p}$. This completes the proof of the claim.
\end{proof}

Getting back to the proof, assume that $J$ is composed of the last $s$ indices of $C$. Then the claim implies that, with probability $1-O(e^{-cn})$, we can write

\[ M/D_{1}' = \begin{pmatrix}
0 &B_{3}  &B_{4} \\ 
B_{3}^{T} &C_{1}  &C_{2} \\ 
B_{4}^{T} &C_{2}^{T}  & C_{3} 
\end{pmatrix}, \]

Where $C_3$ is nonsingular and $\dim(C_3)=s\geq\left(\alpha(1 -\frac{1}{p})-\epsilon\right) n$.

Taking $\epsilon$ to be sufficiently small so that $\alpha\left( \frac{1}{p}+\epsilon\right)<\frac{1}{p}-\epsilon$, we can assume that with probability $1-O(e^{-cn})$, 
\[r>\left(\frac{1}{p}-\epsilon\right)n>\alpha \left( \frac{1}{p}+\epsilon\right) n>\alpha n - s.\]

Since $C_{1}$ is $(\alpha n-s)\times (\alpha n-s)$ and $\height(B_{3})=r$, we can assume that $\width(B_{3})=\dim(C_{1})=\alpha n-s<r=\height(B_{3})$.

Note that we can drop rows and columns from $C_3$ if necessary, thus increasing the width of $B_{3}$, up to a maximum of $\width(B_{3})+\width(B_{4})=\alpha n> r$. In particular, we can assume that $s=\alpha n-r$, so that $B_{3}$ is an $r\times r$ square matrix.

We now wish to shows that $\rank(M/D_{1}')$ is usually within small distance of zero. To do this, we will split the rows into three sets, and successively show that that most of the rows are independent: 

First, let $u_{1},\dots,u_{\alpha n-r}$ be the bottom $\alpha n-r$ rows (those with elements in $C_{3}$). Since they contain as subrows the rows of $C_{3}$ (which we know are independent), they are independent.

Secondly, let $v_{1},\dots,v_{r}$ be the top $r$ rows. By Corollary~\ref{cor:rectrank}, $\corank(B_{3})$ is usually within small distance of zero. In fact, we can make a stronger claim: We claim that the corank of the $\alpha n\times\alpha n$ matrix $\begin{pmatrix}
B_{3} & B_{4}\\ 
C_{2}^{T} & C_{3}
\end{pmatrix}$ is usually within small distance of zero.

Let $v_{i}'$ be the top $r$ rows of this matrix, and $u_{i}'$ be the bottom $\alpha n-r$ rows. As before, the $u_{i}'$ are independent since their tails are the rows of $C_{3}$.

Now assume that the first $k$ of the $v_{i}'$ are independent both of each other and of the $u_{i}'$ (that is, the set $\{u_{1}',\dots,u_{\alpha n-r}',v_{1}',\dots,v_{k}'\}$ is independent. We claim that the probability that $v_{k+1}$ is dependent on $\{u_{1}',\dots,u_{\alpha n-r}',v_{1}',\dots,v_{k}'\}$ is at most $(1-\beta)^{r-k}$.

To see this, first choose a set $J$ of $\alpha n-r+k$ indices so that $\{u_{1}',\dots,u_{\alpha n-r}',v_{1}',\dots,v_{k}'\}$ are still independent when restricted to the entries in $J$. If $v_{k+1}'$ is dependent on $\{u_{1}',\dots,u_{\alpha n-r}',v_{1}',\dots,v_{k}'\}$, then we can write $v_{k+1}'=\sum_{i\leq k} a_{i}v_{i}'+\sum b_{i}u_{i}'$, where the $a_{i}$ and the $b_{i}$ are determined by the entries in $J$. This leaves $n-r$ undetermined coefficients in $v_{k+1}$, all of which must be equal to the corresponding entry of $\sum_{i\leq k} a_{i}v_{i}'+\sum b_{i}u_{i}'$. 

But the entries of $v_{k+1}$ all have min-entropy at least $\beta$ with respect to the other vectors, so each of them is equal to the corresponding entry of  $\sum_{i\leq k} a_{i}v_{i}'+\sum b_{i}u_{i}'$ with conditional probability at most $1-\beta$, hence the probability that all $r-k$ of them satisfy this equality is at most $(1-\beta)^{r-k}$. From here, we can conclude that the corank of the matrix is usually within small distance of zero by following the same reasoning as the proof of Theorem~\ref{thm:rectrank}.

Finally, it remains to show that the middle $r$ rows of $M/D_{1}'$, labeled $w_{1},\dots,w_{r}$, cannot add much to the corank. That is, we need to find a set of independent rows of $M/D_{1}'$ whose size is usually within small distance of $r+\alpha n$. We will assume that the $u_{i}'$ and $v_{i}'$ are all independent (otherwise we only have to drop $k$ of them, where $k$ is usually within small distance of zero).

We proceed in a similar manner to before.  For the first $w_{1}$, we let $J$ be the set of the last $\alpha n$ indices. Since the $u_{i}'$ and $v_{i}'$ are all independent, there exists a unique set of indices $a_{i},b_{i}$ so that $w_{1}=\sum a_{i}u_{i}+\sum b_{i}v_{i}$ holds when restricted to the last $\alpha n$ indices. As the first $r$ entries of $w_{1}$ have min-entropy at least $\beta>0$ with respect to the rest of the matrix, they all match the corresponding entries of $\sum a_{i}u_{i}+\sum b_{i}v_{i}$ with probability at most $(1-\beta)^{r}$.

We proceed similarly, showing that for each $w_{k+1}$ such that the set $\{u_{1},\dots,u_{\alpha n-r},v_{1},\dots,v_{r},w_{1},\dots,w_{k}\}$ is independent, the probability that $w_{k+1}$ is dependent on is is at most $(1-\beta)^{r-k}$. As before, this shows that the number of independent $w_{i}$ is usually within small distance of $r$.

Putting this all together, we get a set of independent rows whose size is usually within small distance of the height of $M/D_{1}'$. The corank of $M/D_{1}'$ is at most the number of rows not in our set, which is usually within small distance of zero. This completes the proof.

\end{proof}

\section{Proofs of the Corollaries}
\label{sec:details}
In this section, we prove the corollaries of Theorem~\ref{thm:main}.

We begin by proving Corollary~\ref{cor:balanced}:

\begin{proof}[Proof of Corollary~\ref{cor:balanced}]
Let $\epsilon>0$, and let $X=X(n,p)$ be the $p$-rank of $G$. We need to show that as $n\rightarrow\infty$, $\E(X)<\epsilon n$.

Assume that $\epsilon<\frac{1}{2}$, and remove $\frac{\epsilon}{2} n$ vertices from the right side of the graph. By Theorem~\ref{thm:main}, the expected $p$-rank of the resulting graph is $O(1)$. Since removing a vertex changes the $p$-rank of the sandpile group by at most $1$, removing $\frac{\epsilon}{2} n$ vertices changes it by at most $\frac{\epsilon}{2} n$. Hence $X\leq \frac{\epsilon}{2}n+O(1)<\epsilon n$ for large $n$, which completes the proof.
\end{proof}

We now prove Corollary~\ref{cor:cyclic}. To do this, we show that the $2$-rank of $\Gamma(G)$ when $\alpha<\frac{1}{2}$ has low probability of being $\leq 1$, so the $2$-part of the group has low probability of being cyclic.

\begin{proof}[Proof of Corollary~\ref{cor:cyclic}]
Consider the $2-$rank of $\Gamma(G)$. As we saw in Theorem~\ref{thm:strongMain}, the $2$-rank of $\Gamma(G)$ is usually within small distance of $\max\left(B\left(n,\frac{1}{2}\right)-\alpha n,0\right)$. As $\alpha<\frac{1}{2}$, we have that by Hoeffding's inequality, 
\[B\left(n,\frac{1}{2}\right)>\left(\frac{1}{2}-\epsilon\right)n>\alpha n+\epsilon n\]
holds with probability $1-O(e^{-cn})$ for all $\epsilon>0$, where the second inequality will hold when $\epsilon<\frac{1}{2}\left(\frac{1}{2}-\alpha\right)$. Hence $\max\left(B\left(n,\frac{1}{2}\right)-\alpha n,0\right)>\epsilon n$ with probability $1-O(e^{-cn})$.

But the $2$-rank of $\Gamma(G)$ is usually within small distance of $\max\left(B\left(n,\frac{1}{2}\right)-\alpha n,0\right)$. Hence the $2$-rank of $\Gamma(G)$ is larger than $\frac{1}{2}\epsilon n$ with probability $1-O(e^{-cn})$ for some $c>0$, and in particular will be at least $2$ with probability $1-O(e^{-cn})$.

But if the $2$-rank of $\Gamma(G)$ is at least $2$, $\Gamma(G)$ cannot be cyclic. Hence the probability that $\Gamma(G)$ is cyclic is bounded by $O(e^{-cn})$ for some constant $c>0$.
\end{proof}

\bibliography{pPart}
\bibliographystyle{plain}

\end{document}